\documentclass{kms-c}

\issueinfo{}
  {}
  {}
  {}
\pagespan{1}{}
\copyrightinfo{}
  {Korean Mathematical Society}

\usepackage{graphicx}
\allowdisplaybreaks

\theoremstyle{plain}
\newtheorem{theorem}{Theorem}[section]
\newtheorem{proposition}[theorem]{Proposition}
\newtheorem{lemma}[theorem]{Lemma}
\newtheorem{corollary}[theorem]{Corollary}

\theoremstyle{definition}
\newtheorem*{definition}{Definition}
\newtheorem{example}[theorem]{Example}

\theoremstyle{remark}

\begin{document}

\title[ Hopf-type hypersurfaces on Hermite-like manifolds]
{Hopf-type hypersurfaces on Hermite-like manifolds}

\author[M. G\"{u}lbahar]{Mehmet G\"{u}lbahar}
\address{Mehmet G\"{u}lbahar \\ Department of Mathematics \\ Harran University \\ \c{S}anlıurfa, 63300, Turkey}
\email{mehmetgulbahar@harran.edu.tr}

\thanks{This study was supported by the Scientific and Technological Research Council of Turkey (TÜBİTAK) with project number 122F326.}


\begin{abstract}
The object of this paper is to introduce new classes of hypersurfaces of almost product-like
statistical manifolds. The main properties and relations on $K-$para contact, para cosymplectic, para Sasakian and conformal hypersurfaces are obtained. Some examples of these hypersurfaces are presented.
\end{abstract}

\maketitle

\section{Introduction}

Studies on the problem of investigating the geometry of Hermitian manifolds include many relations involving almost complex and almost contact structures. Using these relations, various properties can be emerged. Such hypersurfaces can be expressed in general terms as follows:

Let $(M,g)$ be a non-degenerate real hypersurface of a Hermitian manifold \\
$(\widetilde{M}^{0},\widetilde{g}^{0},J)$ such that $g$ is the induced metric from $\widetilde{g}^{0}$ and $J$ is an almost complex structure. Symbolize $N$ is the unit vector field of $(M,g)$. The vector field  $\xi=-JN$ is identified as the structure vector field of $(M,g)$.

A real hypersurface of $(\widetilde{M}^{0},\widetilde{g}^{0},J)$ is defined a Hopf hypersurface if $\xi$ is principal vector field, namely, $A_{N}^{0}\xi=\rho^{0}\xi$, where $A_{N}^{0}$ indicates the shape operator and $\rho^{0}$ is a smooth function on $(M,g)$. In \cite{Chen,Ki,Maeda}, the authors investigated Hopf hypersurfaces having constant principal curvatures. The geometry of these hypersurfaces was discussed in various space forms in \cite{Des,Kimura,Kon,Lee,P,Wang}, etc.

In this study, we introduce Hopf hypersurfaces on Hermite-like statistical manifolds. We investigate concurrent and geodesic vector fields on these hypersurfaces and obtain some characterizations with the aid of statistical connections. We obtain some results dealing with the value of $\theta$, which is a $1$-form given in $(\ref{13})$. We also compute the Riemannian curvature tensor fields and obtain some relations involving sectional curvatures on Hopf hypersurfaces.

The differences between this study and other studies in the literature are as follows:

\begin{itemize}
\item  The definition of Hopf hypersurfaces on the hypersurfaces of Hermitian manifolds is provided by means of the shape operator that emerged with the famous Gauss and Weingarten formulas. In the statistical manifolds, two different shape operators are defined. This leads to quite different relations and results, as demonstrated in this study.

\item The results given in this study were obtained by calculating the $1$-form $\theta$ and Riemann curvature tensors defined by the Gauss and the Weingarten formulas. However, considering the Levi-Civita connection for any hypersurface of Riemannian or Hermitian manifolds, the  $1$-form $\theta$ is not defined.
\end{itemize}

\section{Preliminaries}
Let $\widetilde{M}$ be a $C^{\infty}-$ manifold and $J$ be an almost complex structure on $\widetilde{M}$ verifying $J^{2}=-I$. Here $I$ is the identity map. If there subsists a semi-Riemannian metric $\widetilde{g}$ on $\widetilde{M}$ verifying
\begin{align} \label{1}
 \widetilde{g}(JX_{1},X_{2})=-\widetilde{g}(X_{1},JX_{2})
\end{align}
for any $X_{1},X_{2} \in \Gamma(T\widetilde{M})$, then the manifold is defined an almost Hermitian manifold.
A semi-Riemannian manifold admitting two almost complex structures $J$ and $J^{\ast}$ is called a Hermite-like manifold if the relation
\begin{align} \label{2}
 \widetilde{g}(JX_{1},X_{2})=-\widetilde{g}(X_{1},J^{\ast}X_{2})
\end{align}
is satisfied for any $X_{1},X_{2} \in \Gamma(T\widetilde{M})$. A Hermite-like manifold is usually indicated by $(\widetilde{M}, \widetilde{g},J)$.
For more details, we refer to \cite{Takano:06,Takano:10}.

Let $\widetilde{\nabla}$ be an affine connection on $(\widetilde{\nabla}, \widetilde{g},J)$. If $\widetilde{\nabla}\widetilde{g}$ is symmetric, then \\
$(\widetilde{M},\widetilde{g},\widetilde{\nabla},J)$ is called a Hermite-like statistical manifold. For any Hermite-like statistical manifold, we get
\begin{align} \label{3}
 \widetilde{g}(\widetilde{\nabla}_{X_{3}}X_{1},X_{2})=X_{3} \widetilde{g}(X_{1},X_{2})-\widetilde{g}(\widetilde{\nabla}^{\ast}_{X_{3}}X_{2},X_{1})
\end{align}
for any $X_{1},X_{2},X_{3} \in \Gamma(T\widetilde{M})$ \cite{Amari}. Here, $\widetilde{\nabla}^{\ast}$ is said to be the dual connection of $\widetilde{\nabla}$ and the following relation is satisfied:
\begin{align} \label{4}
    \widetilde{\nabla}^{\circ}=\dfrac{1}{2} ( \widetilde{\nabla}+\widetilde{\nabla}^{\ast}  ),
\end{align}
where $\widetilde{\nabla}^{\circ}$ is the Levi-Civita connection of $(\widetilde{M}, \widetilde{g})$.

Let us indicate the Riemannian curvatures with respect to $\widetilde{\nabla}$ and $\widetilde{\nabla}^{\ast}$ by $\widetilde{R}$ and $\widetilde{R}^{\ast}$. It is known that there exists the following relation between $\widetilde{R}$ and $\widetilde{R}^{\ast}$ for any $Y_{1},Y_{2},Y_{3},Y_{4} \in \Gamma(T\widetilde{M})$:
\begin{align}
\widetilde{g}(\widetilde{R}^{\ast}(Y_{1},Y_{2})Y_{3},Y_{4})=-\widetilde{g}(\widetilde{R}(Y_{1},Y_{2})Y_{4},Y_{3}).
\end{align}
For more details related to statistical manifolds, we refer to \cite{Aydin,Aytimur,Fur:1,Fur:2,Sid}

A Hermite-like statistical manifold is called a Kaehler-like statistical manifold if $\widetilde{\nabla}J=0$ satisfies. We note that $\widetilde{\nabla}^{\ast}J^{\ast}=0$ is also satisfied  for any Kaehler-like statistical manifold.

Let $(M,g)$ be a non-degenerate hypersurface of $(\widetilde{M},\widetilde{g},\widetilde{\nabla},J)$ and $N$ be the unit vector field of $(M,g)$. If $JN$ and $J^{\ast}N$ lie on $\Gamma(TM)$, then $(M,g)$ is defined a tangential hypersurface \cite{Gulbahar}. For any tangential hypersurface, we write $JN=-\xi$, $J^{\ast}N=-\xi^{\ast}$ and
\begin{align}
    JX_{1}&=\varphi X_{1}+\varepsilon \eta^{\ast}(X_{1})N, \label{5} \\
    J^{\ast}X_{1}&=\varphi^{\ast} X_{1}+\varepsilon \eta(X_{1})N, \label{6}
\end{align}
for any $X_{1} \in \Gamma(TM)$, where $\varepsilon=g(\xi, \xi^{\ast})$, $\eta(X_{1})=g(X_{1},\xi)$ and $\eta^{\ast}(X_{1})=g(X_{1},\xi^{\ast})$. We note that the following relations occur for any $Y_{1},Y_{2} \in \Gamma(T\widetilde{M})$:
\begin{align}
    g(\varphi Y_{1},Y_{2})&=-g(Y_{1},\varphi^{\ast} Y_{2}), \label{7} \\
    g(\varphi Y_{1}, \varphi Y_{2})&=g(Y_{1},Y_{2})-\varepsilon  \eta(Y_{1}) \eta^{\ast}(Y_{2}), \label{8} 
\end{align}
\begin{align}
    \varphi^{2} Y_{1}&=-Y_{1}+\varepsilon \eta^{\ast}(Y_{1})\xi, \label{9} \\
    (\varphi^{\ast})^{2} Y_{1}&=-Y_{1}+\varepsilon \eta(Y_{1})\xi^{\ast}, \label{11} \\
    \varphi \xi&=\varphi^{\ast}\xi^{\ast}=0. \label{12}
\end{align}

The Gauss and the Weingarten formulae belonging to $(M,g)$ are represented by
\begin{align}
 \widetilde{\nabla}_{X_{1}}X_{2}&={\nabla}_{X_{1}}X_{2}+\varepsilon g(A^{\ast}_{N}X_{1},X_{2})N,  \label{12} \\
 \widetilde{\nabla}_{X_{1}}N&=-A_{X_{1}}N+\varepsilon \theta(X_{1})N, \label{13} \\
\widetilde{\nabla}^{\ast}_{X_{1}}X_{2}&=\nabla^{\ast}_{X_{1}}X_{2}+\varepsilon g(A_{N}X_{1},X_{2})N, \label{14} \\
 \widetilde{\nabla}^{\ast}_{X_{1}}N&=-A^{\ast}_{X_{1}}N-\varepsilon \theta(X_{1})N, \label{15}
\end{align}
where $\nabla$ and $\nabla^{\ast}$ are the induced connections of $\widetilde{\nabla}$ and $\widetilde{\nabla}^{\ast}$, consecutively.
A tangential hypersurface $(M,g)$ is defined totally geodesic with respect to $\nabla$ (resp. $\nabla^{\ast}$) if $A_{N}=0$ (resp. $A^{\ast}N=0$). $(M,g)$ is defined totally umbilical with respect to $\nabla$ (resp. $\nabla^{\ast}$ ) if there subsists a smooth function $\lambda$ satisfying $A_{N}X=\lambda X$ (resp. $A^{\ast}_{N}X=\lambda X$ ).

\begin{example} \label{ex:1}
   Let us discuss a $4$-dimensional semi-Riemannian manifold $\widetilde{M}$ with the undermentioned metric:
\begin{align*}
\widetilde{g}=\begin{bmatrix}
2& 0 & 0 & 0\\
0 & 2 & 0 & 0\\
0 & 0 & -1 & 0\\
0 & 0 & 0 & -1
\end{bmatrix}.
\end{align*}
If we define
\begin{align*}
J=\begin{bmatrix}
 0& 0 & 1 & 0\\
0 & 0 & 0 & 1\\
-1 & 0 & 0 & 0\\
0 & -1 & 0 & 0
\end{bmatrix}
\quad \text{and} \quad
J^{\ast}=\begin{bmatrix}
 0& 0 & -\frac{1}{2} & 0\\
0 & 0 & 0 & -\frac{1}{2}\\
2 & 0 & 0 & 0\\
0 & 2 & 0 & 0
\end{bmatrix},
\end{align*}
then we obtain $(\widetilde{M},\widetilde{g},J)$ is a Hermite-like manifold. Consider a hypersurface determined by
\begin{align*}
    M=\left \{ (x_{1},x_{2},x_{3},0): x_{1},x_{2},x_{3}  \in \mathbb{R} \right \}.
\end{align*}
Then, we obtain the induced metric $g$ on $M$ stated by
\begin{align*}
{g}=\begin{bmatrix}
 2& 0 & 0 \\
0 & 2 & 0 \\
0 & 0 & -1
\end{bmatrix}
\end{align*}
and
\begin{align*}
\Gamma(TM)=\textrm{Span}\{e_{1}=\partial x_{1}, e_{2}=\partial x_{2}, \ e_{3}=\partial x_{3} \},
\end{align*}
\begin{align*}
\Gamma(TM^{\bot})=\textrm{Span}\{e_{4}\equiv N=\partial x_{4}\},
\end{align*}
where $\{\partial x_{i}\}_{i\in \{1,2,3,4\}}$ is the standard basis of $\widetilde{M}$.
By a straightforward computation, we find that $(M,g)$ is a tangential hypersurface such that
$\xi=e_{2}$ and $\xi^{\ast}=-\frac{1}{2}e_{2}$.

\end{example}

\section{Hopf Hypersurfaces}
\begin{definition}
 Let $(M,g)$ be tangential hypersurface of $(\widetilde{M},\widetilde{g},\widetilde{\nabla},J)$. If there exist smooth functions $\rho$ and $\rho^{\ast}$ satisfying
 \begin{align} \label{16}
  A_{N}\xi=\rho \xi
 \end{align}
 and
 \begin{align} \label{17}
 A^{\ast}_{N}\xi^{\ast}=\rho^{\ast}\xi^{\ast},
 \end{align}
 then we call $(M,g)$ a Hopf hypersurface of $(\widetilde{M},\widetilde{g},\widetilde{\nabla},J)$.
\end{definition}

\begin{example}
Consider the hypersurface $(M,g)$ given in Example \ref{ex:1}. From (\ref{3}), we write
\begin{align*}
e_{i}(\widetilde{g}_{jk})=\Gamma_{ij}^{k}\widetilde{g}_{kk}+\Gamma_{ik}^{\ast j}\widetilde{g}_{jj},
\end{align*}
where $\widetilde{\nabla}_{e_{i}}e_{j}=\sum\limits_{s=1}^{4}\Gamma_{ij}^{s}e_{s}$,$\widetilde{\nabla}^{\ast}_{e_{i}}e_{j}=\sum\limits_{s=1}^{4}\Gamma_{ij}^{\ast s}e_{s}$, and $\widetilde{g}_{ij}=\widetilde{g}(e_{i},e_{j})$ for any $i,j,k\in \{1,2,3,4\}$. In this case, we have
\begin{align*}
  \Gamma_{i1}^{1}+\Gamma_{i1}^{\ast 1} &= 0,   &    \Gamma_{i2}^{1}+\Gamma_{i1}^{\ast 2} &=0, &  2\Gamma_{i3}^{1}-\Gamma_{i1}^{\ast 3} &=0, & 2\Gamma_{i4}^{1}-\Gamma_{i1}^{\ast 4} &=0, \\
   \Gamma_{i1}^{2}+\Gamma_{i2}^{\ast 1} &= 0,   &   \Gamma_{i2}^{2}+\Gamma_{i2}^{\ast 2} &= 0, & 2\Gamma_{i3}^{2}-\Gamma_{i2}^{\ast 3} &=0, & 2\Gamma_{i4}^{2}-\Gamma_{i2}^{\ast 4} &=0,\\
   -\Gamma_{i1}^{3}+2\Gamma_{i3}^{\ast 1} &= 0,   &   -\Gamma_{i2}^{3}+2\Gamma_{i3}^{\ast 2} &= 0, & \Gamma_{i3}^{3}+\Gamma_{i3}^{\ast 3} &=0, & \Gamma_{i4}^{3}+\Gamma_{i3}^{\ast 4} &=0, \\
    -\Gamma_{i1}^{4}+2\Gamma_{i4}^{\ast 1} &= 0,   &   -\Gamma_{i2}^{4}+2\Gamma_{i4}^{\ast 2} &= 0, & \Gamma_{i3}^{4}+\Gamma_{i4}^{\ast 3} &=0, & \Gamma_{i4}^{4}-\Gamma_{i4}^{\ast 4} &=0
\end{align*}
for each $i\in \{1,2,3,4\}$. If we put
\begin{align*}
\widetilde{\nabla}_{e_{2}}e_{4}=\Gamma_{24}^{2}e_{2} \ \ \textrm{and} \ \ \widetilde{\nabla}^{\ast}_{e_{2}}e_{4}=\Gamma_{24}^{\ast 2}e_{2},
\end{align*}
then we get
\begin{align*}
\Gamma_{24}^{1}=\Gamma_{21}^{\ast 4}=\Gamma_{24}^{3}=\Gamma_{23}^{\ast 4}=\Gamma_{24}^{4}=\Gamma_{24}^{\ast 4}=\Gamma_{21}^{4}=\Gamma_{24}^{\ast 1}=\Gamma_{23}^{4}=\Gamma_{24}^{\ast 3}=0.
\end{align*}
Therefore, we obtain $\widetilde{\nabla}_{e_{2}}e_{2}=\Gamma_{22}^{4}e_{4}$, $\widetilde{\nabla}^{\ast}_{e_{2}}e_{2}=\Gamma_{22}^{\ast 4}e_{4}$
and $(M,g)$ is a Hopf hypersurface with $\rho=\frac{1}{2}\rho^{\ast}=\Gamma_{24}^{2}$.

\end{example}

Further examples of Hopf hypersurfaces could be given.

 Let $(M,g)$ be a Hopf hypersurface of $(\widetilde{M},\widetilde{g},\widetilde{\nabla},J)$.Thus, we obtain
\begin{align} \label{18}
  g(A_{N}\xi,\xi^{\ast})= \varepsilon \rho
\end{align}
and
\begin{align}\label{19}
 g(A^{\ast}_{N}\xi^{\ast},\xi)= \varepsilon \rho^{\ast}.
\end{align}
From (\ref{18}) and (\ref{19}), the Gauss and the Weingarten formulae belonging to any Hopf hypersurface can be expressed as
\begin{align}
  \widetilde{\nabla}_{\xi^{\ast}}\xi&= {\nabla}_{\xi^{\ast}}\xi+\rho^{\ast}N, \label{20} \\
  \widetilde{\nabla}^{\ast}_{\xi}\xi^{\ast}&={\nabla}^{\ast}_{\xi}\xi^{\ast}+\rho N, \label{21} \\
  \widetilde{\nabla}_{\xi}N&=-\rho \xi+ \varepsilon \theta(\xi) N, \label{22} \\
  \widetilde{\nabla}^{\ast}_{\xi^{\ast}}N&=-\rho^{\ast} \xi^{\ast}-\varepsilon \theta(\xi^{\ast})N. \label{23}
\end{align}

\begin{proposition}
Let $(\widetilde{M},\widetilde{g},\widetilde{\nabla},J)$ be a Kaehler-like statistical manifold and $(M,g)$ be a  a Hopf hypersurface of $(\widetilde{M},\widetilde{g},\widetilde{\nabla},J)$. The following relations are satisfied:
\begin{align}
 {\nabla}_{\xi^{\ast}}\xi&=-\varphi A_{N}\xi^{\ast}-\varepsilon \theta(\xi^{\ast})\xi, \label{24} \\
 \rho^{\ast}&=-\varepsilon \eta^{\ast}(A_{N}\xi^{\ast}). \label{25}
\end{align}
\end{proposition}
\begin{proof}
Since $(\widetilde{M},\widetilde{g},\widetilde{\nabla},J)$ is a Kaehler-like statistical manifold, we write
\begin{align} \label{26}
  \widetilde{\nabla}_{\xi^{\ast}} JN=J \widetilde{\nabla}_{\xi^{\ast}}N.
\end{align}
Considering (\ref{5}), (\ref{20}), and (\ref{22}) in (\ref{26}), we obtain
\begin{align} \label{27}
    {\nabla}_{\xi^{\ast}}\xi+\rho^{\ast}N=-\varphi A_{N}\xi^{\ast}-\varepsilon \eta^{\ast}(A_{N}\xi^{\ast})N-\varepsilon \theta(\xi^{\ast})\xi.
\end{align}
Taking account of the tangential and normal components of (\ref{27}), we get (\ref{24}) and (\ref{25}).
\end{proof}

In view of (\ref{6}), (\ref{21}), and (\ref{23}), we find
\begin{proposition}
  For any Hopf hypersurface of $(\widetilde{M},\widetilde{g},\widetilde{\nabla},J)$, we get the following relations:
\begin{align}
 {\nabla}^{\ast}_{\xi}\xi^{\ast}&=-\varphi^{\ast}A^{\ast}_{N}\xi+\varepsilon \theta(\xi^{\ast})\xi, \label{28} \\
 \rho&=-\varepsilon \eta(A^{\ast}_{N}\xi). \label{29}
\end{align}
\end{proposition}

\begin{definition} \cite{Yano}
 Let $(\overline{M}, \overline{g})$ be a semi-Riemannian manifold and $\nabla^{\overline{M}}$ be a linear connection on $(\overline{M}, \overline{g})$. If there subsists a vector field $\nu$ that satisfies $\nabla^{\overline{M}}_{X}\nu=X$ for any $X \in \Gamma(T\overline{M})$, then $\nu$ is called a concurrent vector field.
\end{definition}

\begin{theorem} \label{thm3}
  Let $(M,g)$ be a Hopf hypersurface. If $\xi$ is a concurrent vector field with respect to $\nabla$, then $\theta(\xi^{\ast}) \neq 0$.
\end{theorem}
\begin{proof}
 From (\ref{24}) and using the fact that $\xi$ is a concurrent vector field with respect to $\nabla$, we arrive at
\begin{align*}
    \xi^{\ast}=-\varphi A_{N}\xi^{\ast}-\varepsilon \theta(\xi^{\ast})\xi,
\end{align*}
 which denotes that
\begin{align} \label{30}
 g(\xi^{\ast},\xi^{\ast})=-\theta(\xi^{\ast}).
\end{align}
According to (\ref{30}), we find $\theta(\xi^{\ast}) \neq 0$. This 	accomplishes the proving of theorem.
\end{proof}

Using arguments analogue to the proving of Theorem \ref{thm3}, we obtain
\begin{theorem}
 Let $(M,g)$ be a Hopf hypersurface. If $\xi^{\ast}$ is a concurrent vector field with respect to $\nabla^{\ast}$, then $\theta(\xi^{\ast}) \neq 0$.
\end{theorem}

\begin{proposition}
 Let $(M,g)$ be a Hopf hypersurface of Kaehler-statistical manifold $(\widetilde{M},\widetilde{g},\widetilde{\nabla},J,J^{\ast})$. Then the undermentioned relations occur:
 \begin{align} \label{31}
  \widetilde{\nabla}_{\xi^{\ast}} \varphi=\nabla_{\xi^{\ast}} \varphi
 \end{align}
 and
 \begin{align}\label{32}
 \widetilde{\nabla}^{\ast}_{\xi}\varphi^{\ast}= \nabla_{\xi}\varphi^{\ast}.
 \end{align}
\end{proposition}
\begin{proof}
  For any $X \in \Gamma(TM)$, we put
\begin{align}
  (\widetilde{\nabla}_{\xi^{\ast}}\varphi) X&= \widetilde{\nabla}_{\xi^{\ast}}\varphi X-\varphi \widetilde{\nabla}_{\xi^{\ast}}X \nonumber \\
  &=\nabla_{\xi^{\ast}}\varphi X- \varphi \nabla_{\xi^{\ast}}X+\varepsilon g(A^{\ast}_{N}\xi^{\ast},\varphi X)N.
\end{align}
Using the fact that $(M,g)$ is a Hopf hypersurface, we get (\ref{31}). The proving of (\ref{32}) can be derived with an argument similar to (\ref{31}).
\end{proof}

Now, we shall recall the undermentioned proposition in \cite{Gulbahar}:

\begin{proposition}
     For any tangential hypersurface of a Kaehler-like statistical manifold, the below stated relations are satisfied for any $X_{1},X_{2} \in \Gamma(TM)$:
 \begin{align}
   (\nabla_{X_{1}}\varphi)X_{2}&=-\varepsilon g(A^{\ast}_{N}X_{1},X_{2})\xi+\varepsilon \eta^{\ast}(X_{2})A_{N}X_{1}, \label{34} \\
    (\nabla^{\ast}_{X_{1}}\varphi^{\ast})X_{2}&=-\varepsilon g(A_{N}X_{1},X_{2})\xi^{\ast}+\varepsilon \eta(X_{2})A^{\ast}_{N}X_{1}.
    \label{34:a}
 \end{align}
\end{proposition}

\begin{theorem}
 For any Hopf hypersurface of a Kaehler-like statistical manifold, the following situations are true:
 \begin{itemize}
     \item [i)] If $\varphi$ is parallel with respect to $\nabla$ or $\widetilde{\nabla}$, then
     \begin{align} \label{36}
         A_{N}\xi^{\ast}=\rho^{\ast}
     \end{align}
     is satisfied.
     \item[ii)] If $\varphi^{\ast}$ is parallel with respect to $\nabla^{\ast}$ or $\widetilde{\nabla}^{\ast}$, then
     \begin{align} \label{37}
         A^{\ast}_{N}\xi=\rho \xi^{\ast}
     \end{align}
     is satisfied.
 \end{itemize}
\end{theorem}
\begin{proof}
    Substituting $X=\xi^{\ast}$ and $\nabla_{\xi^{\ast}} \varphi=0$ into (\ref{34}), the proof of (\ref{36}) is straightforward. Substituting $X=\xi$ and $\nabla^{\ast}_{\xi}{\varphi^{\ast}}=0$ into (\ref{34:a}), the proof of (\ref{37}) is straightforward.
\end{proof}

Reach back to the definition of geodesic vectors which was firstly introduced by  S. Deshmukh and B.Y. Chen \cite{Des:1,Des:2}, we may state as follows:

\begin{definition}
    A vector field $X$ on a statistical manifold $(\widetilde{M},\widetilde{g},\widetilde{\nabla})$ is defined a geodesic vector field with respect to $\widetilde{\nabla}$ (or $\widetilde{\nabla}^{\ast}$) if there subsists a smooth function $\zeta$  satisfying $\widetilde{\nabla}_{X}X=\zeta X$ (resp. $\widetilde{\nabla}^{\ast}_{X}X=\zeta X$). If $\zeta=0$,
    then $X$ is defined a unit geodesic vector field.
\end{definition}

Now, we shall give the undermentioned proposition in \cite{Gulbahar}.
\begin{proposition}
 For any tangential hypersurface of a Kaehler-like statistical manifold $(\widetilde{M},\widetilde{g}, J, J^{\ast})$, the below stated equalities are satisfied for any $X,Y \in \Gamma(TM)$:  
 \begin{align} \label{38}
   \varepsilon g(A^{\ast}_{N}X,\varphi Y)=-\varphi g(Y, \nabla^{\ast}_{X}\xi^{\ast})-\eta^{\ast}(Y)\theta(X)
 \end{align}
 and
 \begin{align} \label{39}
   \varepsilon g(A_{N}X,\varphi^{\ast} Y)=-\varphi g(Y, \nabla_{X}\xi)+\eta(Y)\theta(X).
 \end{align}
\end{proposition}

\begin{theorem} \label{thm9}
  Let $(M,g)$ be a Hopf hypersurface of a Kaehler-like statistical manifold. Then $\xi^{\ast}$ is a geodesic vector field with respect to $\nabla^{\ast}$ if and only if $\theta(\xi^{\ast})=0$.
\end{theorem}
\begin{proof}
    Using (\ref{17}) and (\ref{38}), we get
\begin{align}\label{40}
g(\nabla^{\ast}_{\xi^{\ast}}\xi^{\ast},\xi)=\theta(\xi^{\ast}),
\end{align}
which yields $\xi$ as a geodesic vector field with respect to $\nabla^{\ast}$. Then, $\theta(\xi^{\ast})=0$.

Now, we shall prove the converse part: Assume that $\theta(\xi^{\ast})=0$ is satisfied. In view of (\ref{23}), we put
\begin{align} \label{41}
\widetilde{\nabla}^{\ast}_{\xi^{\ast}}N=\widetilde{\nabla}^{\ast}_{\xi^{\ast}} J^{\ast}\xi^{\ast}=J^{\ast}\widetilde{\nabla}^{\ast}_{\xi^{\ast}}{\xi^{\ast}}=-\rho^{\ast}\xi^{\ast}.
\end{align}
Using (\ref{6}), (\ref{20})-(\ref{23}) in (\ref{41}), we get that
\begin{align}\label{42}
 \varphi^{\ast}{\nabla}^{\ast}_{\xi^{\ast}}{\xi^{\ast}}-\varepsilon g(A_{N}\xi^{\ast},\xi) \xi^{\ast}+\varepsilon \eta({\nabla}^{\ast}_{\xi^{\ast}}{\xi^{\ast}}) \xi^{\ast}=-\rho^{\ast}\xi^{\ast}.
\end{align}
Taking the action of $\varphi^{\ast}$ on both sides of (\ref{42}), it brings out that
\begin{align}\label{43}
 {\nabla}^{\ast}_{\xi^{\ast}}{\xi^{\ast}}=\varepsilon \eta({\nabla}^{\ast}_{\xi^{\ast}}{\xi^{\ast}}){\xi^{\ast}}.
\end{align}
Substituting (\ref{40}) into (\ref{43}), we derive ${\nabla}^{\ast}_{\xi^{\ast}}{\xi^{\ast}}=0$, which shows that $\xi^{\ast}$ is a geodesic vector field with respect to $\widetilde{\nabla}^{\ast}$.
\end{proof}

With arguments similar to those in the proving of Theorem \ref{thm9}, one can find
\begin{theorem}
  Let $(M,g)$ be a Hopf hypersurface of a Kaehler-like statistical manifold. Thus, $\xi$ is a geodesic vector field with respect to $\nabla$ if and only if $\theta(\xi)=0$.
\end{theorem}

\section{Riemannian Curvature Tensor Fields on Hopf Hypersurfaces}

A Kaehler-like statistical manifold $(\widetilde{M},\widetilde{g},\widetilde{\nabla},J)$ is called space of constant
holomorphic sectional curvature $c$ if 
\begin{eqnarray}
\widetilde{R}(X,Y)Z&=&\frac{c}{4}\left\{\widetilde{g}(Y,Z)X-\widetilde{g}(X,Z)Y-
\widetilde{g}(Y,JZ)JX+\widetilde{g}(X,JZ)JY \right.  \nonumber \\
&&\left. +\widetilde{g}(X,JY)JZ-\widetilde{g}(JX,Y)JZ\right\}, \label{constant:1}
\end{eqnarray}
where $c$ is constant (cf. \cite{Fur:1}).

Let $\widetilde{M}(c)$ denotes Kaehler-like statistical manifold with constant holomorphic sectional curvature $c$. If $(M,g)$ is a tangential hypersurface of $\widetilde{M}(c)$, we get the below stated relations involving the Riemannian curvature tensor $R$ of $(M,g)$:
\begin{align}
  R(X,Y)Z&=\dfrac{c}{4}\left\{ g(Y,Z)X-g(X,Z)Y-g(Y,\varphi Z)\varphi X+g(X, \varphi Z) \varphi Y \right. \nonumber \\
   & \quad \left. +g(X,\varphi Y)\varphi Z-g(X,Y)\varphi Z \right \}+\varepsilon g(A^{\ast}_{N}Y,Z)A_{N}X \nonumber \\
  & \quad -\varepsilon g(A^{\ast}_{N}X,Z)A_{N}Y, \label{44} \\
   R^{\ast}(X,Y)Z&=\dfrac{c}{4}\left\{g(Y,Z)X-g(X,Z)Y-g(Y,\varphi^{\ast}Z) \varphi^{\ast}X+g(X,\varphi^{\ast} Z)\varphi^{\ast} Y \right. \nonumber \\
  & \quad \left. +g(X,\varphi^{\ast}Y)\varphi^{\ast}Z-g(X,Y)\varphi^{\ast} Z \right\}+\varepsilon g(A_{N}Y,Z)A^{\ast}_{N}X \nonumber \\
  & \quad -\varepsilon g(A_{N}X,Z)A^{\ast}_{N}Y, \label{45}
\end{align}
\begin{align}
  (\overline{\nabla}^{\ast}_{X}A^{\ast})_{N}Y-(\overline{\nabla}^{\ast}_{Y}A^{\ast})_{N}X&=\dfrac{c}{4}\left\{ \eta^{\ast}(X)\varphi^{\ast}Y-\eta^{\ast}(Y)\varphi^{\ast}X+g(X,\varphi Y)\xi^{\ast} \right. \nonumber \\
  & \left. \quad -g(X,\varphi Y)\xi^{\ast}\right \}, \label{46} \\
  (\overline{\nabla}_{X}A)_{N}Y-(\overline{\nabla}_{Y}A)_{N}X&=\dfrac{c}{4}\left\{ \eta(X)\varphi Y-\eta(Y)\varphi X +g(X,\varphi Y)\xi\right. \nonumber \\
  & \left. \quad -g(\varphi X,Y)\xi \right \} , \label{47} \\
  g(A_{N}X,A^{\ast}_{N}Y)-g(A_{N}Y,A^{\ast}_{N}X)&=\dfrac{c}{4}\left\{ \eta(Y)\eta^{\ast}(X)-\eta(X)\eta^{\ast}(Y) \right. \nonumber \\
   & \left. \quad (\overline{\nabla}_{X}\theta)(Y)+(\overline{\nabla}_{Y}\theta)(X)\right\}, \label{48}
\end{align}
where
\begin{align*}
(\overline{\nabla}^{\ast}_{X}A^{\ast})_{N}Y&=\nabla^{\ast}_{X}(A^{\ast}_{N}Y)+\varepsilon \theta(X) A^{\ast}_{N}Y-A^{\ast}_{N}(\nabla^{\ast}_{X}Y), \\
 (\overline{\nabla}_{X}A)_{N}Y&=\nabla_{X}(A_{N}Y)+\varepsilon \theta(X) A_{N}Y-A_{N}(\nabla_{X}Y)
\end{align*}
and
\begin{align*}
    (\overline{\nabla}_{X}\theta)(Y)=X\left\{ \theta(Y)\right\}-\theta(\nabla_{X}Y).
\end{align*}

\begin{theorem}
  Let $(M,g)$ be a Hopf hypersurface of $\widetilde{M}(c)$ satisfying $\nabla_{X}\varphi=0$ for any $X \in \Gamma(TM)$. Thus, at least one of the undermentioned situations arises:
  \begin{itemize}
      \item[i)] $c=0$. 
      \item[ii)] The equation
      \begin{align} \label{49}
          (\overline{\nabla}_{\xi^{\ast}}\theta)(\xi)-(\overline{\nabla}_{\xi}\theta)(\xi^{\ast})=\eta(\xi)\eta^{\ast}(\xi^{\ast})-\varepsilon
      \end{align}
      is satisfied.
  \end{itemize}
\end{theorem}
\begin{proof}
 Under the assumption, if $\nabla_{X}\varphi=0$ is satisfied, then we write $A_{N}\xi^{\ast}=\rho^{\ast}\xi$ and $A^{\ast}\xi=\rho \xi^{\ast}$. From (\ref{48}), we find that
    \begin{align}\label{50}
        0=\dfrac{c}{4}\left\{ \eta(\xi)\eta^{\ast}(\xi^{\ast})-\varepsilon-(\overline{\nabla}_{\xi^{\ast}}\theta)(\xi)+(\overline{\nabla}_{\xi}\theta)(\xi^{\ast})  \right\}.
    \end{align}
  The equation (\ref{50}) implies that $c=0$ or (\ref{49}) is satisfied. Thus, the proof of theorem is completed.
\end{proof}

Suppose that $dim M=2n+1$ and $\left\{e_{1},e_{2},\dots,e_{2n-1},\xi,\xi^{\ast} \right\}$ is a basis on $\Gamma(TM)$ such that $\left\{e_{1},e_{2},\dots,e_{2n-1} \right\}$ is an orthonormal set. Thus, we can write
\begin{align*}
    TM=\mathbb{D}_{0}\oplus \mathbb{D}_{1}
\end{align*}
such that $\mathbb{D}_{0}=\textrm{span} \left\{e_{1},e_{2},\dots,e_{2n-1} \right\}$ and $\mathbb{D}_{1}=\textrm{span}\left\{\xi,\xi^{\ast}\right\}$.

\begin{lemma}
    For any Hopf hypersurface of $\widetilde{M}(c)$, we get
  \begin{align}
   g(R(\xi,\xi^{\ast})\xi^{\ast},\xi)&=\dfrac{c}{4}\left\{g(\xi,\xi)g(\xi^{\ast},\xi^{\ast})-1 \right\}+2g^{2}(\varphi^{\ast}\xi,\xi^{\ast}) \nonumber \\
   & \quad +\varepsilon \rho \rho^{\ast} g(\xi^{\ast},\xi^{\ast})g(\xi,\xi)-\varepsilon g(A^{\ast}_{N}\xi,\xi^{\ast})g(A_{N}\xi^{\ast},\xi), \label{51} \\
   g(R^{\ast}(\xi,\xi^{\ast})\xi^{\ast},\xi)&=\dfrac{c}{4}\left\{g(\xi,\xi)g(\xi^{\ast},\xi^{\ast})-1 \right\}+\varepsilon g(A_{N}\xi^{\ast},\xi^{\ast})g(A^{\ast}_{N}\xi,\xi) \nonumber \\
    & \quad -\varepsilon \rho \rho^{\ast}, \label{52} \\
    g(R(X,\xi)\xi,X)&=\dfrac{c}{4}g(\xi,\xi)+\varepsilon g(A^{\ast}_{N}\xi,\xi)g(A_{N}X,X), \label{53} \\
    g(R(X,\xi^{\ast})\xi^{\ast},X)&=\dfrac{c}{4}\left\{g(\xi^{\ast},\xi^{\ast})+2g^{2}(\varphi \xi^{\ast},X) \right\}+\varepsilon \rho^{\ast} g(\xi^{\ast},\xi^{\ast})g(A_{N}X,X) \nonumber \\
    & \quad -\varepsilon g(A^{\ast}_{N}X,\xi^{\ast}) g(A_{N}\xi^{\ast},X), \label{54}
\end{align}
\begin{align} 
    g(R^{\ast}(X,\xi)\xi,X)&=\dfrac{c}{4} \left\{ g(\xi,\xi)+2g^{2}(\varphi^{\ast} \xi,X) \right\}+\varepsilon \rho g(\xi,\xi)g(A^{\ast}_{N}X,X) \nonumber \\
    & \quad -\varepsilon g(A_{N}X,\xi)g(A^{\ast}_{N}\xi,X), \label{55} \\
    g(R^{\ast}(X,\xi^{\ast})\xi^{\ast},X)&=\dfrac{c}{4}g(\xi^{\ast},\xi^{\ast})+\varepsilon g(A_{N}\xi^{\ast},\xi^{\ast})g(A^{\ast}_{N}X,X) \label{56}
 \end{align}
 for any $X \in \mathbb{D}_{0}$.
\end{lemma}

\begin{theorem}
    For any Hopf hypersurface of $\widetilde{M}(c)$, we get
  \begin{eqnarray}
   A_{1} \kappa(\pi_{1}) & \geq &\frac{c}{4} \left\{ g(\xi,\xi)g(\xi^{\ast},\xi^{\ast})-1 \right\}+\varepsilon \rho \rho^{\ast} g(\xi^{\ast},\xi^{\ast}) g(\xi,\xi) \nonumber \\
   && \quad -\varepsilon g(A^{\ast}_{N}\xi,\xi^{\ast})g(A_{N}\xi^{\ast},\xi), \label{57}
  \end{eqnarray}
  where $\pi_{1}=span\left\{\xi,\xi^{\ast} \right\} $ and $ A_{1}=g(\xi,\xi)g(\xi^{\ast},\xi^{\ast})-\varepsilon$. The equality case of (\ref{57}) is held if and only if $J$ is anti-invariant on $\mathbb{D}_{1}$.
\end{theorem}
\begin{proof}
  The proof of (\ref{57}) is obvious from (\ref{51}). For the equality case, we have $g(\varphi^{\ast}\xi,\xi^{\ast})=0$. This implies that
$J\xi^{\ast}$ is perpendicular to $\xi$. Considering this fact with $J\xi=\varepsilon N$, we obtain $J$ is anti-invariant on $\mathbb{D}_{1}$. This completes the proof.
\end{proof}

\begin{theorem} \label{thm14}
   For any Hopf hypersurface of $\widetilde{M}(c)$, we have 
   \begin{eqnarray}
   A_{2} \kappa(\pi_{2}) & \geq & \frac{c}{4}g(\xi^{\ast},\xi^{\ast})+\varepsilon \rho^{\ast}g(\xi^{\ast},\xi^{\ast})g(A_{N}X,X) \nonumber \\
   && \quad -\varepsilon g(A^{\ast}_{N}X,\xi^{\ast})g(A_{N}\xi^{\ast},X) \label{58}
  \end{eqnarray}
for any $X \in \mathbb{D}_{0}$. Here, $A_{2}=g(X,X)g(\xi^{\ast},\xi^{\ast})$ and $\phi_{2}=span\{X, \xi^{\ast}\}$. If the equality case of (\ref{58}) is true for each $X \in \mathbb{D}_{0}$, then $\widetilde{M}(c)$ is a statistical complex space form.
\end{theorem}

\begin{proof}
  From (\ref{54}), the proof of (\ref{58}) is easy to obtain. In the equality case of (\ref{58}), we get $g(\varphi \xi^{\ast}, X)=0$. Thus, we find $\varphi \xi^{\ast} \in \mathbb{D}_{1}$. Considering this fact, we can write
\begin{align}\label{59}
 \varphi \xi^{\ast}=\mu_{1}\xi+\mu_{2}\xi^{\ast}
\end{align}
for the smooth functions $\mu_{1}$ and $\mu_{2}$. Then, from (\ref{9}) and (\ref{58}), we have
  \begin{align*}
    -\xi^{\ast}+\varepsilon \eta^{\ast}(\xi^{\ast})\xi=\lambda_{2}\varphi \xi^{\ast},
  \end{align*}
 it follows that
 \begin{align} \label{60}
  \varphi \xi^{\ast}=-\dfrac{1}{\lambda_{2}}\xi^{\ast}+\dfrac{\varepsilon}{\lambda_{2}}\eta^{\ast}(\xi^{\ast})\xi.
 \end{align}
 Taking into account of (\ref{58}) and (\ref{59}), we find $\lambda_{2}^{2}=-1$, which is a contradiction. Therefore, we obtain $\varphi \xi^{\ast}=0$. This indicates that
 \begin{align*}
  J\xi^{\ast}=\varepsilon \eta^{\ast}(\xi^{\ast})N.
 \end{align*}
Considering the fact that it is $J\xi=\varepsilon N$, we arrive at
 \begin{align} \label{61}
   \xi^{\ast}=\varepsilon \eta^{\ast}(\xi^{\ast}) \xi .
 \end{align}
 According to (\ref{60}), it brings out that $g(\xi,\xi)g(\xi^{\ast},\xi^{\ast})=1$. Now, we accept that $\alpha$ is the angle between $\xi$ and $\xi^{\ast}$. Then we write
\begin{align*}
 cos \alpha=\frac{g(\xi,\xi^{\ast})}{ \left\| \xi \right\|  \left\| \xi^{\ast}\right\|} =1,
\end{align*}
which indicates that $\xi=\varepsilon \xi^{\ast}$. This result  contradicts the fact that $(M,g)$ is a tangential hypersurface. Thus, we find $J=J^{\ast}$. Therefore, $\widetilde{M}(c)$ is a statistical complex space form.
\end{proof}

With arguments similar to those in the proof of Theorem \ref{thm14} and using (\ref{55}), we arrive at

\begin{theorem}
 For any Hopf hypersurface of $\widetilde{M}(c)$, we have
\begin{eqnarray}
   A_{3} \kappa^{\ast}(\pi_{3}) \geq \frac{c}{4}g(\xi,\xi)+\varepsilon \rho g(\xi,\xi)g(A^{\ast}_{N}X,X) -\varepsilon g(A_{N}X,\xi)g(A^{\ast}_{N}\xi,X) \label{62}
  \end{eqnarray}
for any $X \in \mathbb{D}_{0}$. Here, $\pi_{3}=span\left\{X,\xi\right \}$ and $A_{3}=g(\xi,\xi)g(X,X)$. If the equality case of (\ref{62}) is satisfied for each $X \in \mathbb{D}_{0}$, then $\widetilde{M}(c)$ is a statistical complex space form.
\end{theorem}

\begin{lemma}
Let $(M,g)$ be a Hopf hypersurface satisfying $\nabla_{X}\varphi=0$. Thus, we get the undermentioned relations:
\begin{align}
    g(R(\xi,\xi^{\ast})\xi^{\ast},\xi)&=\dfrac{c}{4}\left\{ g(\xi,\xi)g(\xi^{\ast},\xi^{\ast})-1 \right\}+2g^{2}(\varphi^{\ast}\xi,\xi^{\ast}), \label{63} \\
    g(R^{\ast}(\xi,\xi^{\ast})\xi^{\ast},\xi)&=\dfrac{c}{4}\left\{ g(\xi,\xi)g(\xi^{\ast},\xi^{\ast})-1 \right\}, \label{64} \\
    g(R(X,\xi)\xi,X)&=\dfrac{c}{4}g(\xi,\xi)+\rho g(A_{N}X,X), \label{65} 
\end{align}
\begin{align}
    g(R(X,\xi^{\ast})\xi^{\ast},X)&=\dfrac{c}{4}\left\{ g(\xi^{\ast},\xi^{\ast})+2g^{2}(\varphi \xi^{\ast},X) \right\} \nonumber \\
    &\quad +\varepsilon \rho^{\ast}g(\xi^{\ast},\xi^{\ast})g(A_{N}X,X), \label{66} \\
    g(R^{\ast}(X,\xi)\xi,X)&=\dfrac{c}{4}\left\{g(\xi,\xi)+2g^{2}(\varphi^{\ast}\xi,X)\right\} \nonumber \\
    &\quad +\varepsilon \rho g(\xi,\xi)g(A^{\ast}_{N}X,X), \label{67} \\
    g(R^{\ast}(X,\xi^{\ast})\xi^{\ast},X)&=\dfrac{c}{4}g(\xi^{\ast},\xi^{\ast})+\rho^{\ast}g(A^{\ast}_{N}X,X). \label{68}
\end{align}
\end{lemma}

Let us indicate the sectional curvature maps of $(M,g)$ with respect to $\nabla$ and $\nabla^{\ast}$
by $\kappa$ and $\kappa^{\ast}$. In view of (\ref{63}), we get the undermentioned corollaries:
\begin{corollary}
 Let $(M,g)$ be a Hopf hypersurface of $\widetilde{M}(c)$ satisfying $\nabla_{X}\varphi=0$. Then
 \begin{eqnarray}
     A_{1}\kappa(\pi_{1})\geq \frac{c}{4}\left\{ g(\xi,\xi)g(\xi^{\ast},\xi^{\ast})-1)\right\}. \label{69}
 \end{eqnarray}
 If the equality case of (\ref{69}) is satisfied, then $J$ is anti-invariant on $\mathbb{D}_{1}$.
\end{corollary}

\begin{corollary}
 Let $(M,g)$ be a Hopf hypersurface of $\widetilde{M}(0)$ satisfying $\nabla_{X}\varphi=0$. Then we have the following situations:
 \begin{itemize}
     \item[i)] If $g(\xi,\xi)g(\xi^{\ast},\xi^{\ast}) > \varepsilon$ is satisfied, then the sectional curvature map on $\mathbb{D}_{1}$ is positive semi-defined.
     \item[ii)] If $g(\xi,\xi)g(\xi^{\ast},\xi^{\ast}) < \varepsilon$ is satisfied, then the sectional curvature map on $\mathbb{D}_{1}$ is negative semi-defined.
 \end{itemize}
\end{corollary}

In view of (\ref{64}), we find
\begin{corollary}
    Let $(M,g)$ be a Hopf hypersurface satisfying $\nabla_{X}\varphi=0$. If $\mathbb{D}_{1}$ is a space-like distribution, then $\kappa^{\ast}$ is constant on $\mathbb{D}_{1}$ such that
    \begin{align*}
    \kappa^{\ast}(\pi_{1})=\dfrac{c}{4},
    \end{align*}
where $\pi_{1}=span\left\{ \xi, \xi^{\ast} \right\}$.
\end{corollary}

In view of (\ref{65}), we find 

\begin{proposition}\label{prop20}
  Let $(M,g)$ be a Hopf hypersurface satisfying $\nabla_{X}\varphi=0$. If $\mathbb{D}_{0}$ is totally geodesic distribution with respect to $\widetilde{\nabla}$, then we get
    \begin{align*}
      \kappa(\pi_{3})=\dfrac{c}{4g(X,X)},
    \end{align*}
 where $\pi_{3}=span\left\{ \xi,X\right\}$.
\end{proposition}
\begin{proof}
    Under the assumption, if we write $A_{N}=0$ in (\ref{65}), then we find
    \begin{align} \label{70}
       g(R(X,\xi)\xi,X)=\dfrac{c}{4}g(\xi,\xi).
    \end{align}
  From (\ref{70}), we get
  \begin{align*}
     \kappa(\pi_{3})=\dfrac{g(R(X,\xi)\xi,X)}{g(X,X)g(\xi,\xi)}=\dfrac{c}{4g(X,X)},
  \end{align*}
  which is the proof of proposition.
\end{proof}

As a consequence of Proposition \ref{24}, we find the undermentioned corollary:

\begin{corollary}
  Let $(M,g)$ be a Hopf hypersurface satisfying $\nabla_{X}\varphi=0$ and let $\mathbb{D}_{0}$ be a totally geodesic distribution with respect to $\widetilde{\nabla}$. Then $\kappa(\pi_{3})=0$ if and only if $c=0$. 
\end{corollary}

In view of (\ref{66}), we find
\begin{corollary}
 Let $(M,g)$ be a Hopf hypersurface satisfying $\nabla_{X}\varphi=0$. Then we have
 \begin{eqnarray}\label{71}
   A_{2}\kappa(\pi_{2}) \geq  \frac{c}{4}g(\xi^{\ast},\xi^{\ast})+\varepsilon \rho^{\ast} g(\xi^{\ast},\xi^{\ast}) g(A_{N}X,X),
 \end{eqnarray}
where $\pi_{2}=span\left\{ X, \xi^{\ast} \right\}$.
If the equality case of (\ref{71}) is satisfied, then $\widetilde{M}(c)$ is a statistical complex space form.
\end{corollary}

\begin{corollary}
    Let $(M,g)$ be a Hopf hypersurface satisfying $\nabla_{X}\varphi=0$. Then we have
    \begin{align} \label{72}
      \kappa^{\ast}(\pi_{3}) \geq \frac{c}{4}g(\xi,\xi)+\varepsilon \rho g(\xi,\xi)g(A^{\ast}_{N}X,X),
    \end{align}
 where $\pi_{3}=span\left\{ \xi,X\right\}$. If the equality case of (\ref{72}) is satisfied, then $\widetilde{M}(c)$ is a statistical complex space form.
\end{corollary}

With arguments similar to the proof of Proposition \ref{prop20}, we find
\begin{proposition} \label{prop24}
    Let $(M,g)$ be a Hopf hypersurface satisfying $\nabla_{X}\varphi=0$. If $\mathbb{D}_{0}$ is totally geodesic distribution with respect to $\widetilde{\nabla}$, then we have
    \begin{align}
        \kappa(\pi_{2})=\dfrac{c}{4g(X,X)} .
    \end{align}
\end{proposition}

As a result of Proposition \ref{prop24}, we get the undermentioned corollary:
\begin{corollary}
    Let $(M,g)$ be a Hopf hypersurface satisfying $\nabla_{X} \varphi=0 $ and let $\mathbb{D}_{0}$ be a totally geodesic distribution with respect to $\widetilde{\nabla}^{\ast}$. Then $\kappa(\pi_{2})=0$ if and only if $c=0$.
\end{corollary}


\begin{thebibliography}{99}

\bibitem{Amari} S. Amari, {\it Differential-Geometrical Methods in Statistics; Lecture Notes in Statistics},
Springer, New York, 1985.


\bibitem{Aydin} M. E. Aydin, A. Mihai and I.  Mihai, 
{\it Some inequalities on submanifolds in statistical manifolds of constant curvature}
Filomat \textbf{29} (2015), no. 3, 465--477.


\bibitem{Aytimur}
H. Aytimur, M. Kon, A. Mihai, C. \"{O}zg\"{u}r and K. Takano,
{\it Chen inequalities for statistical submanifolds of {K}\"{a}hler-like statistical manifolds},
Mathematics \textbf{7} (2019), no. 12, 1202.


\bibitem{Chen}
B. Y. Chen and S. Maeda,
{\it Hopf hypersurfaces with constant principal curvatures in complex projective or complex hyperbolic spaces},
Tokyo J. Math. \textbf{24} (2001), no. 1, 133--152.


\bibitem{Des}
S. Deshmukh and F. Al-Solamy, {\it Hopf hypersurfaces in nearly {K}aehler
  6-sphere}, Balk. J. Geom. Appl. \textbf{13} (2008), no. 1,  38--46.
  
  
 \bibitem{Des:1} S. Deshmukh and B. Y. Chen,
{\it A note on Yamabe solitons}, Balk. J. Geom. Appl.
\textbf{23} (2018), 37--42.
  
\bibitem{Des:2}  
  S. Deshmukh, P. Peska and H. Bin Turki,
  {\it Geodesic vector fields on a Riemannian manifold},
  Mathematics \textbf{8} (2020), no. 1,  137.
  
  
\bibitem{Fur:1} H. Furuhata,
{\it Hypersurfaces in statistical manifolds},
Differ Geom Appl. \textbf{27} (2009), no. 3, 420--429.

 \bibitem{Fur:2} \bysame,
 {\it Statistical hypersurfaces in the space of Hessian curvature zero},
 Differ Geom Appl.  \textbf{29} (2011), 586--590.


\bibitem{Gulbahar}
E.~{Erkan} and M.~{G\"{u}lbahar}, {\it Tangential real hypersurfaces on Hermite-like manifolds},
  or arXiv:2209.00973v1, (2022). 

\bibitem{Ki}
U.~H. Ki and Y.~J. Suh, {\it On real hypersurfaces of a complex space form},
  Math. J. Okayama Univ. \textbf{32} (1990),  207--221.
  
 \bibitem{Kimura}
M.~Kimura and M.~Ortega, {\it Hopf real hypersurfaces in the indefinite
  complex projective space}, Mediterr. J. Math.  \textbf{16} (2019), no. 2,
  p.~27.
  
 \bibitem{Kon}
M.~{Kon}, {\it On a {H}opf hypersurface of a complex space form}, 
Differ. Geom. Appl. \textbf{28} (2010), no.~3,  295--300.  
  
  \bibitem{Lee}
H.~Lee and S.~Kim, {\it Hopf hypersurfaces with $\eta$-parallel shape operator
  in complex two-plane grassmannians}, Bull. Malays. Math. Sci. Soc.
  \textbf{36} (2013), no.~4,  937--948. 
  
  \bibitem{Maeda}
Y.~Maeda, {\it On real hypersurfaces of a complex projective space}, J.
  Math. Soc. Japan  \textbf{28} (1976),  529--540. 
  
  \bibitem{P}
K.~Panagiotidou and M.~M. Tripathi, 
{\it semi-parallelism of normal jacobi
  operator for hopf hypersurfaces in complex two-plane grassmannians},
  Monatsh. fur Math.  \textbf{172} (2013),  167--178. 
  
  \bibitem{Sid}
  A. N. Siddiqui, B. Y. Chen and  M. D. Siddiqi,
 {\it Chen inequalities for statistical submersions between statistical manifolds},
 Int. J. Geom. Methods Mod. \textbf{18} (2021), 2150049.
 
  
  \bibitem{Takano:06}
K.~Takano, {\it Statistical manifolds with almost contact structures and its
  statistical submersions}, J. Geom. \textbf{85} (2006), no. 1-2,  171--187.
  
  \bibitem{Takano:10}
K.~Takano, {\it Statistical manifolds with almost complex structures},
  Proceedings of the 45-th Symposium on Finsler Geometry, Tokyo,
  2010-09-05/10. Society of Finsler Geometry, pp. 54--57, 2011.

\bibitem{Wang}
Y.~Wang, {\it Three dimensional 2-hopf hypersurfaces with harmonic curvature},
  J. Math. Anal. Appl.   \textbf{499} (2021), no.~1, p. 125005.

\bibitem{Yano}
K.~Yano, {\it Sur le parallelisme et la concourance dans l'espace de riemann},
  Proc. Imp. Acad.  \textbf{19} (1943),  189--197.



\end{thebibliography}
\end{document}